\documentclass[letterpaper,conference]{ieeeconf}
\usepackage[top=53pt, left=54pt, right=54pt, bottom=56pt]{geometry}
\usepackage{amsthm,amsmath,amssymb,mathtools,newtxmath}
\usepackage{graphicx}
\usepackage{cite}
\usepackage[T1]{fontenc}
\usepackage{cleveref}
\usepackage{enumerate}

\newcommand\bigzero{\makebox(0,0){\text{\huge0}}}

\newtheorem{assumption}{Assumption}
\newtheorem{proposition}{Proposition}
\newtheorem{lemma}{Lemma}
\newtheorem{theorem}{Theorem}

\newtheorem{remark}{Remark}
\newtheorem{definition}{Definition}
\newtheorem{corollary}{Corollary}

\usepackage{booktabs,lipsum}
\usepackage{color}

\IEEEoverridecommandlockouts                              

\title{\LARGE \bf The Inverse Problem of Linear-Quadratic Differential Games: When is a Control Strategies Profile Nash? 
}

 \author{Yunhan Huang$^{1}$, Tao Zhang$^{1}$, and Quanyan Zhu$^{1}$
 \thanks{$^{1}$ Y. Huang, T. Zhang, and Q. Zhu are with the Department of Electrical and Computer Engineering, New York University, 370 Jay St., Brooklyn, NY.
 {\tt\small \{yh.huang, tz636, qz494\}@nyu.edu}}
 }

\begin{document}

\maketitle

\begin{abstract}
This paper aims to formulate and study the inverse problem of non-cooperative linear quadratic games: Given a profile of control strategies, find cost parameters for which this profile of control strategies is Nash. We formulate the problem as a leader-followers problem, where a leader aims to implant a desired profile of control strategies among selfish players. In this paper, we leverage frequency-domain techniques to develop a necessary and sufficient condition on the existence of cost parameters for a given profile of stabilizing control strategies to be Nash under a given linear system. The necessary and sufficient condition includes the circle criterion for each player and a rank condition related to the transfer function of each player. The condition provides an analytical method to check the existence of such cost parameters, while previous studies need to solve a convex feasibility problem numerically to answer the same question. We develop an identity in frequency-domain representation to characterize the cost parameters, which we refer to as the Kalman equation. The Kalman equation reduces redundancy in the time-domain analysis that involves solving a convex feasibility problem. Using the Kalman equation, we also show the leader can enforce the same Nash profile by applying penalties on the shared state instead of penalizing the player for other players' actions to avoid the impression of unfairness.


\end{abstract}

\section{Introduction}
The non-cooperative differential game was firstly driven by \cite{isaacs1954differential,isaacs1954differential2,isaacs1955differential4}, involves a set of self-interested players who optimizes their somewhat conflicting objectives over a finite or infinite horizon in a dynamic environment that can usually be described by differential or difference equations. After almost 70 years of development, The theory of non-cooperative differential games has been enriched \cite{bacsar1998dynamic,engwerda2005lq} and applied to many areas such as economics and management science \cite{dockner2000differential}, military operation \cite{weintraub2020introduction,isaacs1955differential4,huang2021defending}, engineering \cite{franzini2016h}, and the modelling and control of epidemics \cite{huang2019differential,huang2022game,reluga2010game}. The most popular solution concept in such games is called Nash equilibrium, which is a profile of strategies where no player can reduce his cost by unilaterally deviating his strategy from it. Characterizing the Nash equilibrium usually involves knowing players' objective functions and applying either dynamic programming or minimum principle to show the optimality of every player's strategy while fixing the strategies of other players \cite{bacsar1998dynamic}.


The inverse problem of differential games consists of characterizing the objective functions (or the parameters that parameterize the objective functions) of individual players based on their observed actions or strategies. The problem has recently caught much attention \cite{molloy2019inverse,inga2019solution,ratliff2012pricing,awasthi2020inverse,lian2022data} due to its application in pricing design \cite{ratliff2012pricing,zhang2021differential}, bonics and humanoid robots \cite{inga2019validation,molloy2018inverse}, and apprentice learning for multi-agent systems \cite{lian2021inverse,lian2022data}.

In this paper, we study the inverse problem of non-cooperative linear-quadratic differential games. The problem is to find the players' objective functions that make a given strategy profile a Nash equilibrium. These objective functions are quadratic in the state and players' controls and parameterized by the cost parameters. We call the cost parameters that make a strategy profile a Nash equilibrium Nash-inducing cost parameters. Previous work leverages the coupled Riccati equation derived from the dynamic programming equation to form a convex feasibility problem \cite{ratliff2012pricing,inga2019solution}. Then, such Nash-inducing cost parameters can be found by numerically solving the convex feasibility problem. However, a fundamental question remains open: when is a control strategies profile a Nash equilibrium? That is, given a strategy profile, whether there exist cost parameters such that the given strategy profile is Nash? The answer to the question should only be decided by the dynamic equations of the players and the given strategy profile. 

In this paper, inspired Kalman's seminal work in inverse optimal control \cite{kalman1964when}, we answer this fundamental question by leveraging frequency-domain techniques. We develop a necessary and sufficient condition for a profile of strategies to be Nash without involving the cost parameters. The convex feasibility problem posed in previous work \cite{ratliff2012pricing,inga2019solution} can then be checked analytically. The necessary and sufficient condition only depends on the given profile of strategies and the dynamic equations of the players. More specifically, the necessary and sufficient condition involves a circle criterion for each player and a rank condition related to the denominator of the transfer function of each player. We also derive an identity in frequency-domain representation, which we refer to as the Kalman equation. The Kalman equation characterizes the cost parameters that make a profile of strategies Nash. Compared with the feasibility conditions derived from the coupled Riccati equation in the time domain, the Kalman equation helps reduce the redundancy in state-space representation. The Kalman equation shows that the leader can implant the same Nash profile by applying penalties on the shared state without penalizing the player for other players' actions, which further reduces the number of cost parameters we need to characterize.

\subsection{Notation}
Let $\mathbb{R}$ be the space of real numbers and
$\mathbb{S}_+$ the set of all real-valued symmetric positive semi-definite matrices. Let
$\mathbb{S}_{++}$ denote the set of real-valued symmetric positive definite matrices.
The identity matrix of dimension $n$ is denoted by $I_n$. Let $\mathbb{C}$ denote the complex plane. Define $\mathbb{C}_{-} \coloneqq \{s\in\mathbb{C}\vert \mathfrak{Re}(s)<0\}$ and $\mathbb{C}_+ = \mathbb{C}-\mathbb{C}_-$, where $\mathfrak{Re}(s)$ is the real part of $s\in\mathbb{C}$. The Kronecker product is denoted by $\otimes$.

\section{Problem Formulation}\label{sec:ProblemForm}

Consider an $N$-player differential game with system dynamics
\begin{equation}\label{Eq:SystemDynamics}
\dot{x}(t) = Ax(t) + \sum_{i=1}^N B_i u_i(t),\ \ \ x(0)=x_0.
\end{equation}
Here, $x$ is the $n$-dimensional state of the system; $u_i$ contains the $m_i$-dimensional variables player $i$ can control; $x_0$ is the initial state of the system (arbitrarily chosen). The system matrix $A$ and the control matrices $B_i$ are real-valued matrices with proper dimension. Suppose there are not redundant control variables, i.e., $B_i$ has rank $m_i$ for every $i$. Denote $\mathcal{N}\coloneqq \{1,2,\cdots,N\}$ the set of players.

The cost criterion or objective function player $i\in\mathcal{N}$ aims to minimize is:
\begin{equation}\label{Eq:CostCriteria}
J_i(u_1,\cdots,u_N,x_0) = \int_{0}^\infty x(t)'Q_ix(t) + \sum_{j\in\mathcal{N}} u_j'(t) R_{ij} u_j(t) dt
\end{equation}
with $Q_i\in\mathbb{S}_{+}^n$ $R_{ii}\in\mathbb{S}_{++}^{m_i}$ for $i,j\in\mathcal{N}$, and $R_{ij}\in\mathbb{S}_{+}^{m_j}$ for $i\neq j$, $i,j\in\mathcal{N}$.
We assume that the players have closed-loop perfect state (CLPS) information pattern \cite[p.~225]{bacsar1998dynamic} and their strategies are stationary and linear in the state, i.e.,  $u_i = K_ix$ for $i\in\mathcal{N}$, where $K_i \in \mathbb{R}^{m_i\times n}$.
\begin{assumption}\label{Assum:Stabilizable}
The system (\ref{Eq:SystemDynamics}) described is stabilizable, i.e., the set 
$$
\mathcal{K}= \left\{ (K_1,\cdots,K_N)\middle\vert A - \sum_{i\in\mathcal{N}} B_i K_i\textrm{ is Hurwitz} \right\}
$$
is non-empty.
\end{assumption}
Since the controls are taking the form $u_i = K_i x$, we can write $J_i$ as a function of $K_i,x_0$.  Now we define the concept of Nash equilibrium under the CLPS information pattern as follows
\begin{definition}[Feedback Nash Equilibrium]
An $N$-tuple $\mathbf{K^*}=(K^*_1,\cdots,K_N^*)$ is called a feedback Nash equilibrium if for every $i$,
$$
J_i(\mathbf{K}^*,x_0) \leq J_i(\mathbf{K}^*_{-i}(K_i),x_0),
$$
for all $x_0\in\mathbb{R}^n$ and for all $K_i$ such that $\mathbf{K}^*_{-i}(K_i) \in \mathcal{K}$, where $\mathbf{K}^*_{-1}(K_i) = (K_1^*,\cdots, K_{i-1}^*,K_i,K_{i+1}^*,\cdots, K_{N}^*)$. 
\end{definition}

To give the inverse problem more context, we suppose there is a leader who has influence on the $N$-player differential game. We refer the $N$ players as followers. A leader's influence on the game is through the choices of cost matrices $Q_i$ and $\{R_{ij}\}_{ij\in\mathcal{N}}$ in (\ref{Eq:CostCriteria}) for $i\in\mathcal{N}$. The goal of the leader is to find cost parameters such that the Nash equilibrium of the game is $\mathbf{K}^\dag = (K_1^\dag,\cdots,K_N^\dag)$, a profile of strategies that the leader wants the followers to adopt. 

\begin{assumption}\label{Assum:StratStable}
The strategy favored by the leader stabilizes the system $(\ref{Eq:SystemDynamics})$, i.e., $\mathbf{K}^\dagger \in \mathcal{K}$.
\end{assumption}

We can define the strategy space of the leader by
$$
\begin{aligned}
\Gamma_0 \coloneqq \{\{Q_i\}_{i\in\mathcal{N}},\{R_{ij}\}_{i,j\in\mathcal{N}}: &Q_i\in\mathbb{S}^n_+, R_{ii}\in \mathbb{S}^{m_i}_{++},i\in\mathcal{N},\\
&R_{ij}\in\mathbb{S}_{+}^{m_i},j\neq i,i,j\in\mathcal{N}\}.
\end{aligned}
$$

The leader announces his strategy $\gamma_0 \in \Gamma_0$ and the followers play the $N$-player differential game defined by (\ref{Eq:SystemDynamics}) and (\ref{Eq:CostCriteria}) with $\{Q_i\}_{i\in\mathcal{N}}$ and $\{R_{ij}\}_{i,j\in\mathcal{N}}$ given by the leader's announced strategy $\gamma_0$. We assume that followers are rational and play a Nash equilibrium. Anticipating that the followers play a Nash equilibrium, the leader aims to find $\gamma_0\in\Gamma_0$ such that the Nash equilibrium of the follower game is $\mathbf{K}^\dag$.

\begin{remark}
How the leader chooses $\mathbf{K}^\dag\in\mathcal{K}$ depends on applications. Since our result applies to any stabilizing $\mathbf{K}^\dag\in\mathcal{K}$, we skill the discussion of how to choose $\mathbf{K}^\dag$ and assume $\mathbf{K}^\dag$ is given. Note that the result can be also extended to the partial observation scenario. To study this case, we can simply let $K_i^\dag = \tilde{K}_i^\dag C_i$.
\end{remark}

In this paper, we address the leader's problem by answering the following fundamental questions: given $\mathbf{K}^\dag$, does there exist $\gamma_0\in\Gamma_0$ such that the Nash equilibrium of the follower game is $\mathbf{K}^\dag$? It is ideal to answer the existence question only using the profile of strategies $\mathbf{K}^\dag$ and the system dynamics $(A,[B_1,B_2,\cdots,B_N])$ without explicitly finding a tuple of $\{Q_{i}\}_{i\in\mathcal{N}}$, $\{R_{ij}\}_{i,j\in\mathcal{N}}$. Before we address the above questions in the next section, we review some useful preliminary results.


\begin{theorem}{\cite[p.~337]{bacsar1998dynamic},\cite[Theorem 4]{engwerda2000feedback}}\label{Theo:LQGame}
Suppose there exist $N$ symmetric matrices $P_i\in\mathbb{S}^n$, $i\in\mathcal{N}$ such that the algebraic Riccati equations (AREs)
\begin{equation}\label{Eq:AREs}
Q_i + P_i A_{cl} + A_{cl}' P_i +  \sum_{j\in\mathcal{N}} P_j B_j R_{jj}^{-1} R_{ij}R_
{jj}^{-1}B_j'P_j =0
\end{equation}
hold for $i\in\mathcal{N}$ with $A_{cl}\coloneqq A - \sum_{i\in\mathcal{N}}B_iR_{ii}^{-1}B_i'P_i$ being Hurwitz.  Define $K^*_i$ as
$$
K_i^* = R_{ii}^{-1} B_i' P_i.
$$
Then, $\mathbf{K}^* = (K_1^*,\cdots,K_N^*)$ constitutes a Nash equilibrium and $J_i(x_0,\mathbf{K}^*)= x_0'P_ix_0$. Conversely, if $\mathbf{K}^* = (K_1^*,\cdots,K_N^*)$ is a Nash equilibrium, the set of AREs (\ref{Eq:AREs}) has a stabilizing solution.
\end{theorem}

Theorem \ref{Theo:LQGame} presents a sufficient and necessary condition for characterizing the Nash equilibrium for the follower's differential game.

\section{Main Results}\label{Sec:TheoreticalResults}

Given a target strategy that the leader aims to install in the followers $\mathbf{K}^\dag\in\mathcal{K}$, define a set $\Theta_{\mathbf{K}^\dag}$ whose elements are the tuples
$$(Q_1,\cdots,Q_N,R_{11},\cdots,R_{NN},P_1,\cdots,P_N)$$
that satisfy the following constraints
\begin{equation}\label{Eq:FeasibilitySet}
\begin{aligned}
    Q_i + P_i {A_{cl}^\dag} + {A_{cl}^\dag}' P_i + \sum_{j\in\mathcal{N}} {K_j^\dag}'R_{ij} K_j^\dag &= 0,\\
    R_{ii} K^\dag_i &= B_i' P_i,\\
    R_{ii} \succ 0,
    R_{ij} & \succeq 0, j\neq i\\
    Q_i &\succeq 0,\\
    P_i & \succeq 0,
\end{aligned} 
\end{equation}
for $i\in\mathcal{N}$, where $A_{cl}^\dagger= A -\sum_{i\in\mathcal{N}} B_i K_i^\dag $. 
\begin{proposition}\label{Prop:NonEmptySet}
Given a target strategy $\mathbf{K}^\dag$ satisfying Assumption \ref{Assum:StratStable}, i.e., $\mathbf{K}^\dag \in\mathcal{K}$, $\mathbf{K}^\dag$ is a Nash equilibrium of the follower game defined by (\ref{Eq:SystemDynamics}) and (\ref{Eq:CostCriteria}) under some $\{\{Q_i\}_{i\in\mathcal{N}},\{R_{ij}\}_{ij\in\mathcal{N}}\}$ if and only if $\Theta_{\mathbf{K}^\dag}$ is non-empty.
\end{proposition}

Proposition \ref{Prop:NonEmptySet} is a direct result of applying Theorem \ref{Theo:LQGame}. One can find $\{\{Q_i\}_{i\in\mathcal{N}},\{R_{ij}\}_{ij\in\mathcal{N}}\}$ that renders $\mathbf{K}^\dag$ a Nash equilibrium by finding the feasibility set $\Theta_{\mathbf{K}^\dag}$ of (\ref{Eq:FeasibilitySet}). The following lemma shows that (\ref{Eq:FeasibilitySet}) is indeed a convex feasibility problem.

\begin{lemma}
\begin{enumerate}
    \item If $\theta \in \Theta_{\mathbf{K}^\dag}$, for any given $\alpha >0$, $\alpha \theta \in \Theta_{\mathbf{K}^\dag}$.
    \item The feasible set $\Theta_{\mathbf{K}^\dag}$ is convex.
\end{enumerate}
\end{lemma}

\begin{remark}
The inverse problem is relevant to many application domains such as mechanism design \cite{ratliff2012pricing,zhang2021differential}, adversarial manipulation \cite{huang2019deceptive,huang2022reinforcement}, apprentice learning \cite{lian2021inverse}. Their problem formulations usually center around the feasibility problem of (\ref{Eq:FeasibilitySet}). The following are several examples.

\textbf{Mechanism Design:} Suppose that the leader aims to design the cost parameters such that the associated Nash equilibrium achieves a value close to the social welfare. Then, the leader's problem can be formulated as:
\begin{equation}\label{Eq:MechanismDesign}
\begin{aligned}
&\min_{\mathbf{K}^\dag}\ \ \ \sum_{i=1}^N J_i(\mathbf{K}^\dag,x_0) - C_o^*\\
&s.t.\ \ \ \textrm{(\ref{Eq:FeasibilitySet}) is feasible for every $i$},
\end{aligned}
\end{equation}
where $C_o^*$ is defined as the value of the optimal control problem: $\min_{\mathbf{K}} \sum_{i=1}^N J_i(\mathbf{K},x_0)$.

\textbf{Adversarial Manipulation:} Consider the leader as an adversary who aims to implant a nefarious policy $\mathbf{K}^\dag$ through manipulating the cost parameters and have the manipulated cost parameters stay as close as possible to the true cost parameters. Then the problem can formulated as
$$
\begin{aligned}
&\min_{\{Q_i,P_i\}_{i\in\mathcal{N}},\{R_{ij}\}_{i,j\in\mathcal{N}}} \sum_{i} \Vert Q_i - Q_i^o \Vert + \sum_{i,j} \Vert R_{ij} - R^o_{ij}\Vert\\
&\ \ \ \ \ \ \ \ s.t.\ \ \ \ \ \ \ \ \ \ \  \textrm{(\ref{Eq:FeasibilitySet}) for every }i\in\mathcal{N}.
\end{aligned}
$$
The solution of such optimization problem gives an attack strategy in reinforcement learning to manipulate the learned strategy to the desired strategy $\mathbf{K}^\dag$. Such an attack strategy is effective especially when the cost parameters need to be estimated using cost data \cite{huang2019deceptive,huang2022reinforcement,ma2019policy}.

\textbf{Multi-Agent Apprentice Learning:} The leader has a sampled (noisy) demonstrations from selfish experts who play Nash. The goal is to find the Nash strategies directly from the sampled demonstrations $(\hat{x}[1],\hat{x}[2], \cdots, \hat{x}[Z])$ and $(\hat{u}_i[1],\hat{u}_i[2], \cdots, \hat{u}_i[Z])_{i\in\mathcal{N}}$, where $\hat{x}[z]=x(z\cdot\Delta t)+\eta_x$ and $\hat{u}_i[z]=u_i(z\cdot \Delta t)+\eta_{u_i}$. Here, $\Delta t$ is the sampling period, and $\eta_x$ and $\eta_{u_i}$ are the noise induced from observations. Then we can formulate the multi-agent apprentice learning problem as 
$$
\begin{aligned}
&\min_{\mathbf{K}^\dag}\ \ \ \sum_{i=1}^N\sum_{z=1}^Z\Vert K^\dag_i \hat{u}_i[z] - \hat{x}[z] \Vert\\
&s.t.\ \ \ \ \ \textrm{(\ref{Eq:FeasibilitySet}) is feasible for every }i\in\mathcal{N}
\end{aligned}
$$
\end{remark}

To solve these inverse problems, the leader needs to numerically solve the convex feasibility problem (\ref{Eq:FeasibilitySet}) to see whether there exists $\gamma_0\in\Gamma_0$ such that $\mathbf{K}^\dag$ is the Nash equilibrium of the followers' game. Apart from the numerical computation, there is also redundancy in (\ref{Eq:FeasibilitySet}) that requires extra effort to solve the inverse problem.

In Kalman's seminal work on inverse optimal control, he developed an optimality condition in the frequency-domain to answer the question when a given strategy is optimal for a given linear system. He developed the so-called ``circle criterion'' or ``return difference condition'' \cite{kalman1964when} which allows deciding whether a strategy is optimal for some cost parameters without solving the convex feasibility problem for the inverse optimal control problem. The criterion is developed for scalar optimal control. Researchers have extended this result to both discrete-time and continuous-time optimal control \cite{fujii1984complete} and \cite{sugimoto1988solution}. In view of their results, we develop such conditions for a multi-player non-cooperative differential game, which, to the best our knowledge, has not been studied previously.

To facilitate later dicussion, define 
$$
\begin{aligned}
\tilde{A}^\dag_i &= A - \sum_{j\neq i} B_j K_j^\dag,\\
\tilde{Q}^\dag_i &= Q_i + \sum_{j\neq i} {K^\dag_j}' R_{ij} K_j^\dag .
\end{aligned}
$$

The first constraint of (\ref{Eq:FeasibilitySet}) can be reconstructed using $\tilde{A}_i^\dag$ and $\tilde{Q}_i^\dag$:
\begin{equation}\label{Eq:ReconsRiccati}
    \tilde{Q}_i^\dag + P_i A_i^\dag + {A_i^\dag}'P_i - P_i B_i R_{ii}^{-1} B_i'P_i =0.
\end{equation}

\subsection{Cases when $R_{ii}=I$ and $R_{ij}=0$}\label{Eq:CaseR=I}

Suppose that the leader only has access to the costs associated with the shared state, i.e., the leader can only alter the cost parameters $\{Q_i\}_{i\in\mathcal{N}}$. Without loss of generality, we let $R_{ii}=I_{m_i}$ and $R_{ij} =0$ for $j\neq i$.


By Assumption \ref{Assum:StratStable}, we know that for every $i\in\mathcal{N}$, $(\tilde{A}_i^\dag, B_i)$ is stabilizable. For the system $(\tilde{A}_i^\dag, B_i)$ for each $i\in\mathcal{N}$, let's consider the following pair of right-coprime polynomial matrices $(S_i(s),D_i(s))$:
\begin{equation}\label{Eq:CoprimeFac}
(sI_n - \tilde{A}_i^\dag)^{-1}B_i = S_i(s) D_i(s)^{-1},
\end{equation}
where $D_i(s)$ is column reduced\footnote{The definition of column reduction is omitted due to space limitation. Readers can refer to \cite[Definition~ 7.5]{chen1999linear}.}. Follower $i$'s feedback $u_i = K_i^\dag x_i$ converts the system $(\tilde{A}_i^\dag, B_i)$ to $(A_{cl}^\dag, B_i)$. The latter induces a right-coprime factorization:
$$
(sI_n -{A}^\dag_{cl})^{-1} B_i = S_i(s) \tilde{D}_{i}(s)^{-1},
$$
where $\tilde{D}_i(s) = D_i(s) + K_i S_i(s)$.

\begin{proposition}\label{Prop:KuceraIdentity}
Let $R_{ii}=I_{m_i}$ and $R_{ij}=0$ for all $i,j\in \mathcal{N},j\neq i$ and $\mathbf{K}^\dag$ satisfy Assumption \ref{Assum:StratStable}. The constraints (\ref{Eq:FeasibilitySet}) are feasible if and only if the following identity holds:
\begin{equation}\label{Eq:KuceraIdentity}
\tilde{D}_i'(-s) \tilde{D}_i(s)  =  D_i'(-s)D_i(s) + S_i'(-s)Q_iS_i(s),
\end{equation}
for every $i\in\mathcal{N}$.
\end{proposition}

\begin{proof}
Adding and subtracting $sP_i$ from (\ref{Eq:ReconsRiccati}) yields 
$$
(-sI_n -{A_i^\dag}')P_i + P_i(sI_n- A_i^\dag) =  \tilde{Q}_i^\dag - P_iB_i R_{ii}^{-1}B_i'P_i.
$$
Then, pre-multiplying it by $B_i'(-sI_n - {A_i^\dag}')^{-1}$, post-multiplying it by $(sI_n -A_i^\dag)^{-1}B_i$, and using the second equation of (\ref{Eq:FeasibilitySet}) and (\ref{Eq:CoprimeFac}), we obtain 
$$
\begin{aligned}
&D'_i(-s)R_{ii} K_i^\dag S_i(s) + S_i'(s) {K_i^\dag}' R_{ii}  D_i(s)\\
&= S_i'(-s)\left[\tilde{Q}_i^\dag -P_iB_i R_{ii}^{-1}B_i'P_i\right]S_i(s).
\end{aligned}
$$
Adding $D_i'(-s) R_{ii} D_i(s)$ to both sides of the equation above yields
\begin{equation}\label{Eq:GenricKuceraIdentity}
\begin{aligned}
&\left[ D_i'(-s) + S_i'(-s){K_i^\dag}' \right] R_{ii} \left[D_i(s) + K_i^\dag S_i(s) \right] \\
&=S_i'(-s) \tilde{Q}^\dag_i S_i(s) + D_i'(-s) R_{ii} D_i(s).
\end{aligned}
\end{equation}
When $R_{ii}=I_{m_i}$ and $R_{ij}=0$ for $i,j\in\mathcal{N},j\neq i$, the above identity becomes (\ref{Eq:KuceraIdentity}).

Conversely, since $A_{cl}^\dag$ is stable by Assumption \ref{Assum:Stabilizable}, and $\tilde{Q}_i^\dag + {K_i^\dag}' R_{ii} K_i^\dag$ is positive semi-definite, there exists a solution $P_i\succeq 0$ to the Lyapunov function
\begin{equation}\label{Eq:LyapunovFun}
P_iA_{cl}^\dag + {A_{cl}^\dag}'P_i = -\tilde{Q}_i^\dag  - {K_i^\dag}' R_{ii} K_i^\dag.
\end{equation}
Note that (\ref{Eq:ReconsRiccati}) can be written as
$$
\tilde{Q}_i^\dag + P_iA_{cl}^\dag + {A_{cl}^\dag}'P_i= P_iB_i R_{ii}B_i'P_i - P_iB_iK_i^\dag - {K_i^\dag}'B_i'P_i.
$$
Hence, it suffices to show $R_{ii} K_i^\dag = B_i'P_i$.
Rewrite (\ref{Eq:LyapunovFun}) as
$$
P_i(sI_n-A_{cl}^\dag) + (-sI_n-{A_{cl}^\dag}')P_i = \tilde{Q}_i^\dag  + {K_i^\dag}' R_{ii} K_i^\dag.
$$
Post- and pre-multiplying the above identity by $S_i(s)$ and $S_i'(-s)$ yields 
$$
\begin{aligned}
&S_i'(-s)P_iB_i\tilde{D}_i(s) + \tilde{D}_i'(-s) B_i'P_i S_i(s)\\
&=S'_i(-s)\left[ \tilde{Q}_i^\dag  + {K_i^\dag}' R_{ii} K_i^\dag  \right] S_i'(s).
\end{aligned}
$$
Combining the above identity and (\ref{Eq:GenricKuceraIdentity}), we obtain
$$
\begin{aligned}
&S_i'(-s)P_iB_i\tilde{D}_i(s) + \tilde{D}_i'(-s) B_i'P_i S_i(s) \\
=& 2 S_i'(s) {K_i^\dag}' R_{ii} {K_i^\dag} S_i(s) + D_i'(-s)R_{ii}K_i^\dag S_i(s) \\
&+S_i'(-s){K_i^\dag}'R_{ii}D_i(s)\\
=&\tilde{D}_i'(-s) R_{ii}{K_i^\dag}S_i(s) + S_i'(-s){K_i^\dag}'R_{ii}\tilde{D}_i(s).
\end{aligned}
$$
The above identity gives
$$
\small
\begin{aligned}
&S_i'(-s)\left[P_iB_i -{K_i^\dag}'R_{ii}\right] \tilde{D}_i(s) + \tilde{D}_i'(-s)\left[B_i'P_i -R_{ii}K_i^\dag \right] S_i(s)\\
&\equiv 0,
\end{aligned}
$$
which indicates $F(s) = -F'(-s)$ where
$$
F(s)\coloneqq (B_i'P_i -R_{ii}K_{i}^\dag)S_i(s) \tilde{D}^{-1}_i(s).
$$
Indeed, we have $F(s)=-F'(-s)\equiv 0$ due to the fact that all the poles of $F(s)$ are in $\mathbb{C}_-$ and those of $F'(-s)$ in $\mathbb{C}_+$. Hence, 
$$
(B_i'P_i -R_{ii}K_{i}^\dag)S_i(s) \equiv 0.
$$
From \cite[Theorem~4.3]{wolovich2012linear}, $S_i(s)$ can be transformed into
\begin{equation}\label{Eq:SForm}
\footnotesize
 S_i'(s)H_i' = \begin{bmatrix}
1 &  &  &  \\
s &  & \bigzero  &  \\
\vdots &  &  &  \\
s^{\sigma_{i,1}-1} &  &  &  \\
 & 1 &  &  \\
 & s &  &  \\
 & \vdots &  &  \\
 & s^{\sigma_{i,2}-1} &  &  \\
 &  & \ddots &  \\
 &  &  & 1 \\
 & \bigzero &  & s \\
 &  &  & \vdots \\
 &  &  & s^{\sigma_{i,m_i}-1} 
\end{bmatrix}  
\end{equation}
by some non-singular matrix $H_i$, where $\sigma_{i,k}$ is the column degree of the $k$-th column of $D_i(s)$.
Hence, 
$B_i'P_i -R_{ii}K_{i}^\dag$ has to vanish.
\end{proof}
Proposition \ref{Prop:KuceraIdentity} bridges state-space and frequency-domain techniques for the conditions that make $\mathbf{K}^\dag$ Nash. In general, several weights $(Q_1, \cdots,Q_N)$ exist for which the Riccati equation in (\ref{Eq:FeasibilitySet}) holds for given $\mathbf{K}^\dag$. We see that  (\ref{Eq:KuceraIdentity}) reduces this type of redundancy, which is also the original reason of using frequency-domain representations for a scalar optimal control problem by Kalman. Equation (\ref{Eq:KuceraIdentity}) is a generalization of the well-known Kalman equation (See \cite[Eq.~45]{kalman1964when}) to a dynamic game setting.
Hence, we also refer (\ref{Eq:KuceraIdentity}) to as the Kalman equation.

Now the question left is whether there exist some $Q_i\succeq 0,i\in\mathcal{N}$ such that (\ref{Eq:KuceraIdentity}) holds for all $i\in\mathcal{N}$. In (\ref{Eq:KuceraIdentity}), the difference
\begin{equation}\label{Eq:DefinePhi}
\Phi_i(s)\coloneqq \tilde{D}_i'(-s) \tilde{D}_i(s) - D_i'(-s)D_i(s),\ \ \ i\in\mathcal{N}
\end{equation}
is independent of $Q_i,i\in\mathcal{N}$ and decided by $\mathbf{K}^\dag$. Inspired from the circle criterion which is a necessary condition for a linear control to be optimal \cite{kalman1964when,fujii1984complete}, we conjecture that a necessary condition for $\mathbf{K}^\dag$ to be Nash is 
\begin{equation}\label{Eq:CircleCriterion}
\Phi_i(jw)\succeq 0\ \ \ \forall i\in\mathcal{N},\ \forall w\in\mathbb{R},
\end{equation}
where $j=\sqrt{-1}$. Now, we present a necessary and sufficient condition for $\mathbf{K}^\dag$ to be Nash under some $(Q_1,Q_2,\cdots,Q_N)$, which subsumes the game version of circle criterion (\ref{Eq:CircleCriterion}).

\begin{theorem}\label{Theo:FeasiblityCheck}
Given $\mathbf{K}^\dag$, suppose that Assumption \ref{Assum:StratStable} holds. Define $\Phi_i(s)\coloneqq \tilde{D}_i'(-s) \tilde{D}_i(s) - D_i'(-s)D_i(s),\  i\in\mathcal{N}$. If $\Phi_i(s)$ has polynomial rank $p_i<m_i$, there exists an $m_i\times m_i$ unimodular matrix\footnote{Readers can refer to \cite[Definition~7.2]{chen1999linear} for the definition of unimodular matrix and \cite[Definition~7.2]{chen1999linear} and \cite[Definition~7.4]{chen1999linear} for the definition of polynomial degree.} $L_i(s)$ that transforms $\Phi_i(s)$ into
\begin{equation}\label{Eq:TransformPhi}
\Phi_i(s)L_i(s) = \begin{bmatrix}
\tilde{\Phi}_i(s) & 0
\end{bmatrix},
\end{equation}
where $\tilde{\Phi}_i$ is an $m_i\times p_i$ polynomial matrix with rank $p_i$. Then $\mathbf{K}^\dag$ is Nash if and only if both conditions below hold:
\begin{enumerate}[(a)]
    \item The circle criterion holds (\ref{Eq:CircleCriterion});
    \item There does not exist an $s\in\mathbb{C}_+$ and a non-zero $v_i\in\mathbb{R}^{m_i}$ such that
    $$
    D_i(s)L_i(s)v_i =0,\textrm{and }v_{i,1}=\cdots=v_{i,p_i} =0
    $$
    for all $i\in\mathcal{N}$. Here, $v_{i,k}$ is the $k$-th element of vector $v_i$.
\end{enumerate}
\end{theorem}
\begin{proof}
The first statement about the existence of $N_i$ is true by Lemma \ref{Lemma:SpectralFactorization}. Let $N_i(s)$ be such a $p_i\times m_i$ polynomial matrix as in Lemma \ref{Lemma:SpectralFactorization}. Then, there exists a unimodular matrix $L_i(s)$, as is shown in Lemma \ref{Lemma:SpectralFactorization}, such that
$$
N_i(s) L_i(s) = \begin{bmatrix} \hat{N}_i(s) & 0\end{bmatrix},
$$
where $\mathrm{det}\ \hat{N}_i(s)\neq 0$ for $\mathfrak{Re}(s)>0$. Suppose (b) does not. Then, there exists an $s\in\mathbb{C}_+$ and a $v_i\in\mathbb{R}^{m_i}$ such that $D_i(s)L_i(s)v_i =0$ and $v_{i,1}=\cdots=v_{i,p}=0$. By (\ref{Eq:TransformPhi}), $\Phi_i(s)L_i(s)v_i =0$. Suppose that $\mathbf{K}^\dagger$ is Nash under some $Q_i,i\in\mathcal{N}$. By (\ref{Eq:KuceraIdentity}) and defining $\bar{N}_i(s)\coloneqq Q_i^{1/2}S_i(s)$, we have
$$
\Phi_i(s) = \bar{N}_i'(-s)\bar{N}_i(s).
$$
From (\ref{Eq:SpectralFactorizationPhi}), we arrive at
$$
L_i'(-s)\bar{N}_i'(-s)\bar{N}_i(s)L_i(s) = \begin{bmatrix}
\hat{N}_i'(-s)\hat{N}_i(s) & 0\\
0 & 0
\end{bmatrix}.
$$
Factorizing $\bar{N}_i(s)L_i(s)$ into $[\bar{N}_{i,1}(s)\ \bar{N}_{i,2}(s)]$ with $\bar{N}_{i,1}(s)$ having $p_i$ columns. Hence, 
$\bar{N}_{i,2}'(-s)\bar{N}_{i,2}(s)\equiv 0$. Then, $\bar{N}_{i,2}^*(jw)\bar{N}_{i,2}(jw)=0$ for an arbitrary real number $w$. Here, superscript $*$ denotes the conjugate transpose. Hence, $\bar{N}_{i,2}(s)=0$ on $\mathbb{C}$. Since $\Phi_i(s)L_i(s)v_i =0$, $\bar{N}_i(s)L_i(s)v_i=0$, which implies that $(\tilde{A}^\dag_i,Q_i)$ is not detectable. By fact 2.4 of \cite{fujii1984complete}, $K^\dagger_i$ cannot be optimal for system $(\tilde{A}_i^\dagger, B_i)$. Hence, $\mathbf{K}^
\dagger$ cannot be Nash, which is a contradiction. 

Suppose that (b) holds. Then, it holds that
\begin{equation}\label{Eq:RankCriterion}
\mathrm{rank}\ \begin{bmatrix}
N_i(s)\\
D_i(s)
\end{bmatrix}=m_i,\ \ \ \forall s\in\mathbb{C}_+.
\end{equation}
Otherwise, letting $u_i=L_i(s)v_i$, we have $N_i(s)u_i = D_i(s)u_i =0$ for some $s\in\mathbb{C}_+$, which contradicts (\ref{Eq:RankCriterion}). For each $i$, since $(\ref{Eq:RankCriterion})$ holds, $(\tilde{A}_i^\dag,C_i)$ is detectable for the matrix characterized by $N_i(s)=C_iS_i(s)$ and (\ref{Eq:KuceraIdentity}) holds for $Q_i = C_i'C_i$. Then from Proposition \ref{Prop:KuceraIdentity}, (\ref{Eq:FeasibilitySet}) is feasible. Hence, $\mathbf{K}^\dag$ is Nash for some $\{Q_i\}_{i\in\mathcal{N}}$ by Proposition \ref{Prop:NonEmptySet}.
\end{proof}

\begin{remark}
Theorem 2 gives an analytical way to check whether the leader's problem is feasible. We can do it without numerically solving the convex feasibility problem (\ref{Eq:FeasibilitySet}).
Consider a two-player linear-quadratic dynamic game with 
$$
A=\begin{bmatrix}
1 & 0 & 1\\
0 & 0 & 1\\
0 & 1 & 0
\end{bmatrix},\ B_1 = \begin{bmatrix}
1 & 0\\
0 & 1\\
0 & 0
\end{bmatrix},\ B_2 = \begin{bmatrix}
1\\
0\\
0
\end{bmatrix}.
$$
Here, $(A,[B_1, B_2])$ is stabilizable. Suppose that the leader promotes the strategy $\mathbf{K}^\dag =(K_1^\dagger, K_2^\dagger)$ by designing $\{Q_i\}_{i=1,2}$. Suppose that
$$
K^\dag_1 = \begin{bmatrix}
1 & 0 & 1\\
0 & 1+\sqrt{2} & 1+\sqrt{2}
\end{bmatrix},\ \ K^\dag_2 = \begin{bmatrix}
1 & 0 & 0\\
\end{bmatrix}.
$$
Note that $\tilde{A}_1^\dag = A -B_2 K^\dag_2$. We have $(sI-\tilde{A}_1^\dag)^{-1}B_1 = S_1(s)D_1(s)^{-1}$, where
$$
S_1(s) = 
\begin{bmatrix}
1 & 0\\
0 & s\\
0 & 1
\end{bmatrix},\ \ D_1 = \begin{bmatrix}
s & -1\\
0 & s^2-1
\end{bmatrix}.
$$
Note that 
$$
\tilde{D}_1(s) = D_1(s)+K_1S_1(s) = \begin{bmatrix}
s+1 & 0\\
0 & (s+1)(s+\sqrt{2}).
\end{bmatrix}
$$
Then, using (\ref{Eq:DefinePhi}), we arrive at
$$
\Phi_1(jw) = \begin{bmatrix}
1 & -jw\\
jw & w^2
\end{bmatrix},
$$
and $\mathrm{det}\ \Phi_1(jw) =2w^2\geq 0$ for all $w\in\mathbb{R}$. Hence, $\Phi_1(jw)\succeq 0$ for all $w\in\mathbb{R}$, meaning that condition (a) in Theorem \ref{Theo:FeasiblityCheck} holds. A unimodular polynomial matrix $L_i(s)$ such as in (\ref{Eq:TransformPhi}) can be found as
$$
L_1 = \begin{bmatrix}
1 & s\\
- & s
\end{bmatrix}.
$$
Let $v_1 = [0\ 1]'$. Then, $D_1(s)L_1(s)v_1 = [s^2-1\ s^2-1]$, which vanishes at $s=1\in\mathbb{C}_+$. Conditions (b) is violated; hence, there are no cost parameters that make $\mathbf{K}^\dag$ a Nash equilibrium.

\end{remark}

\subsection{The general case}

In Section \ref{Eq:CaseR=I}, the leader only has access to the costs associated with the shared state, i.e., $\{Q_i\}_{i\in\mathcal{N}}$. Now consider the general case where the leader can manipulate not only $\{Q_i\}_{i\in\mathcal{N}}$ but also $\{R_{ij}\}_{i,j\in\mathcal{N}}$. In view of the proof of Proposition \ref{Prop:KuceraIdentity}, we can derive the following corollary:

\begin{corollary}\label{Coro:KuceraIdentityGeneral}
Let $\mathbf{K}^\dagger$ satisfy Assumption  \ref{Assum:StratStable}. The constraints (\ref{Eq:FeasibilitySet}) are satisfied if and only if the following equality holds
\begin{equation}\label{Eq:KuceraIdentityGeneral}
\tilde{D}_i'(-s) R_{ii} \tilde{D}_i(s)  =  D_i'(-s)R_{ii}D_i(s) + S_i'(-s)\tilde{Q}^\dagger_iS_i(s),
\end{equation}
for every $i\in\mathcal{N}$.
\end{corollary}

Equation (\ref{Eq:KuceraIdentityGeneral}) is the Kalman equation (\ref{Eq:KuceraIdentity}) under the general case. Comparing with solving (\ref{Eq:FeasibilitySet}), solving (\ref{Eq:GenricKuceraIdentity}) avoids characterizing $\{P_i\}_{i\in\mathcal{{N}}}$ and involves solving a system of linear equations with elements of $\{Q_i\}_{i\in\mathcal{N}}$ and $\{R_{ij}\}_{i,j\in\mathcal{N}}$ being the unknowns.

\begin{proposition}\label{Prop:RedundantPenality}
Suppose that there exist $\{\tilde{Q}_{i}\}_{i\in\mathcal{N}}$ and $\{\tilde{R}_{ij}\}_{i,j\in\mathcal{N}}$ with $\tilde{R}_{ij}$ being non-zero for some $i\neq j$ such that (\ref{Eq:KuceraIdentityGeneral}) holds for every $i\in\mathcal{N}$. Then, there must exist $\{\bar{Q}_i\}_{i\in\mathcal{N}}$ and $\{\bar{R}_{ij}\}_{i,j\in\mathcal{N}}$ with $\bar{R}_{ij}=0$ for all $i\neq j$ such that (\ref{Eq:KuceraIdentityGeneral}) holds for every $i\in\mathcal{N}$. 

Conversely, suppose that there exist $\{\bar{Q}_i\}_{i\in\mathcal{N}}$ with $Q_{i}\succ 0$ and $\{\bar{R}_{ij}\}_{i,j\in\mathcal{N}}$ with $\bar{R}_{ij}=0$ for all $i\neq j$. Then, there must exist $\{\tilde{Q}_{i}\}_{i\in\mathcal{N}}$ and $\{\tilde{R}_{ij}\}_{i,j\in\mathcal{N}}$ with $\tilde{R}_{ij}$ being non-zero for some $i\neq j$ such that (\ref{Eq:KuceraIdentityGeneral}) holds for every $i\in\mathcal{N}$.
\end{proposition}
\begin{proof}
If under $\{\tilde{Q}_{i}\}_{i\in\mathcal{N}}$ and $\{\tilde{R}_{ij}\}_{i,j\in\mathcal{N}}$ with non-zero $\tilde{R}_{ij}$ for some $i\neq j$, (\ref{Eq:KuceraIdentityGeneral}) holds for every $i\in\mathcal{N}$. Let $\bar{Q}_i = \tilde{Q}_i + \sum_{j\neq i} {K_j^\dagger}' \tilde{R}_{ij} K_j^\dagger$ and $\bar{R}_{ii} = \tilde{R}_{ii}$, and $\bar{R}_{ij}=0$ for $i,j\in\mathcal{N},j\neq i$. Obviously, (\ref{Eq:KuceraIdentityGeneral}) holds for every $i\in\mathcal{N}$ under $\{\bar{Q}_i\}_{i\in\mathcal{N}}$ and $\{\bar{R}_{ij}\}_{i,j\in\mathcal{N}}$. Note that the positive semi-definiteness requirement also holds: $\bar{Q}_i\succeq 0$ since $\tilde{Q}_i\succeq 0$ and $\tilde{R}_{ij}\succeq 0$ for $i,j\in\mathcal{N}$ and $i\neq j$. 

If under $\{\bar{Q}_i\}_{i\in\mathcal{N}}$ with $\bar{Q}_{i}\succ 0$ and $\{\bar{R}_{ij}\}_{i,j\in\mathcal{N}}$ with $\bar{R}_{ij}=0$ for all $i\neq j$, (\ref{Eq:KuceraIdentityGeneral}) holds for every $i\in\mathcal{N}$. For $i,j\in\mathcal{N},j\neq i$, let $\tilde{R}_{ij}$ be any $m_j\times m_j$ positive semi-definite matrix. There must exist a scalar $\lambda >0$ such that $\lambda \bar{Q}_i - \sum_{j\neq i} {K_j^\dagger}' \tilde{R}_{ij}K_k^\dagger$ is positive semi-definite. Let $\tilde{Q}_i = \lambda \bar{Q}_i - \sum_{j\neq i} {K_j^\dagger}' \tilde{R}_{ij}K_k^\dagger$ and $\tilde{R}_{ii}=\lambda\bar{R}_{ii}$ for $i\in\mathcal{N}$. Then, under $\{\tilde{Q}_i\}_{i\in\mathcal{N}}$ and $\{\tilde{R}_{ij}\}_{i,j\in\mathcal{N}}$, (\ref{Eq:KuceraIdentityGeneral}) holds for every $i\in\mathcal{N}$.
\end{proof}

Note that $R_{ij}\neq 0$ for some $j\neq i$ means that players are penalized by not their own control but also other players' controls. In some applications, this mechanism can invoke the perception of unfairness among competitive players who are not willing to be penalized for the controls of their cohort. Proposition \ref{Prop:RedundantPenality} indicates that the leader can enforce the same Nash policy $\mathbf{K}^\dagger$ by applying penalties on the shared state instead of penalizing the player with prices induced other players' controls. The result, hence, can be used for mechanism design to circumvent unfairness.

Together with Corollary \ref{Coro:KuceraIdentityGeneral}, Proposition \ref{Prop:RedundantPenality} shows that if (\ref{Eq:FeasibilitySet}) is satisfied under some $\{Q_{i}\}_{i\in\mathcal{N}}$ and $\{R_{ij}\}_{i,j\in\mathcal{N}}$, where not all $R_{ij},i,j\in\mathcal{N},j\neq i$ are zeros. We can find $\{\tilde{Q}_i\}_{i\in\mathcal{N}}$ and $\{\tilde{R}_{i,j}\}_{i,j\in\mathcal{N}}$ with $\tilde{R}_{ij}=0$ for all $i,j\in\mathcal{N},j\neq i$. Hence, to see whether (\ref{Eq:FeasibilitySet}) is feasible, it is sufficient to focus on  $\{Q_i\}_{i\in\mathcal{N}}$ and $R_{ii}$ and let $R_{ij}=0$ for $j\neq i$. Then, it is sufficient to find $\{Q_i\}_{i\in\mathcal{N}}$ and $\{R_{ii}\}_{i\in\mathcal{N}}$ such that 
\begin{equation}\label{Eq:KuceraIdentityGeneral2}
\tilde{D}_i'(-s) R_{ii} \tilde{D}_i(s)  -  D_i'(-s)R_{ii}D_i(s) = S_i'(-s){Q}_iS_i(s),
\end{equation}
for every $i\in\mathcal{N}$.

Let $\tilde{D}_{R_{ii}}(s) = R^{1/2} \tilde{D}_i(s)$ and $D_{R_{ii}}(s) = R^{1/2} D_i(s)$. Applying the arguments of Theorem \ref{Theo:FeasiblityCheck} to the general case, we know if there exists $R_{ii}$ such that conditions (a) and (b) (with $\tilde{D}_i(s)$ and ${D}_i(s)$ in replace of $\tilde{D}_{R_{ii}}(s)$ and ${D}_{R_{ii}}(s)$ respectively) hold for every $i\in\mathcal{N}$, then $\mathbf{K}^\dagger$ is Nash for some $\{Q_{ii}\}_{i\in\mathcal{N}}$ and $\{R_{ii}\}_{i\in\mathcal{N}}$. Checking conditions for the general case analytically is challenging since one needs to show that there exist some $R_{ii}$ such that conditions (a) and (b) hold. However, the frequency-domain representation gives the Kalman equation (\ref{Eq:KuceraIdentityGeneral2}) for the general case. The Kalman equation (\ref{Eq:KuceraIdentityGeneral2}) provides a system of linear equations for us to solve for $Q_{i}$ and $R_{ii}$ numerically. Comparing with solving the convex feasibility problem (\ref{Eq:FeasibilitySet}), solving (\ref{Eq:KuceraIdentityGeneral2}) avoids the redundancy of the Riccati equation and $R_{ij},j\neq i$.

Indeed, we can also write the first two equalities of (\ref{Eq:FeasibilitySet}) in the form of linear equations. First let's define the Kronecker sum of an $n\times n$ matrix $N$ and an $m \times m$ matrix $M$
$$
N \oplus M = (N \otimes I_m) + (I_n \otimes M),
$$
where $\otimes$ denotes the Kronecker product \cite{bellman1997introduction}.
\begin{proposition}
The convex feasibility problem (\ref{Eq:FeasibilitySet}) has a solution if and only if the following system of linear equations has a solution for every $i\in\mathcal{N}$
\begin{equation}\label{Eq:VectorizedRiccati}
\begin{bmatrix}
I_{n^2} & {K_i^\dag}'\otimes {K_i^\dag}' & {A_{cl}^\dag}' \oplus {A_{cl}^\dag} '\\
0 & {K_i^\dag}'\otimes I_{m_i} &  -I_n\otimes B_i'
\end{bmatrix}\begin{bmatrix}
\mathrm{vec}(Q_i)\\
\mathrm{vec}(R_{ii})\\
\mathrm{vec}(P_i)\\
\end{bmatrix} = 0
\end{equation}
such that $Q_i\succeq 0$, $P_i\succeq 0$, and $R_{ii} \succ 0$.
\end{proposition}
\begin{proof}
In view of Proposition \ref{Prop:RedundantPenality}, it is sufficient to discuss the case when $R_{ij=0}$ for $j\neq i$. Vectorize the first equality of (\ref{Eq:FeasibilitySet}) yields
\begin{equation}\label{Eq:VectorizedRiccati1}
\mathrm{vec}(Q_i) + \mathrm{vec}\left(P_i{A^\dag_{cl}}\right) + \mathrm{vec}\left({A^\dag_{cl}}'P_i\right)+\mathrm{vec}\left({K_i^\dag}'R_{ii}K_{i}^\dag\right) =0.
\end{equation}
Note that the following equality holds for any matrices $M$, $V$, and $N$ with proper dimensions \cite{bellman1997introduction}
\begin{equation}\label{Eq:KroneckerProperty}
\mathrm{vec}(MVN) = (N' \otimes M)\mathrm{vec}(V).
\end{equation}
Applying (\ref{Eq:KroneckerProperty}) in (\ref{Eq:VectorizedRiccati1}) yields
$$
\begin{aligned}
\mathrm{vec}(Q_i) &+ \left[ \left({A_{cl}^\dag}'\otimes I_n\right) +\left(I_n\otimes {A_{cl}^\dag}'\right) \right] \mathrm{vec}(P_i) \\
&+ \left( {K_i^\dag}'\otimes {K_i^\dag} \right) \mathrm{vec}(R_{ii}) = 0.
\end{aligned}
$$
Similarly, the second equality of (\ref{Eq:FeasibilitySet}) can be vectorized as
$$
({K_i^\dag}'\otimes I_{m_i}) \mathrm{vec}(R_{ii}) = \left( I_n\otimes B_i' \right)\mathrm{vec}(P_i).
$$
Combining the two equations above yields (\ref{Eq:VectorizedRiccati}).
\end{proof}
\begin{remark}
One can also solve (\ref{Eq:VectorizedRiccati}) instead of (\ref{Eq:FeasibilitySet}) because the equivalence between the two. We see that even if we leverages the Kalman equation to remove the redundancy in $R_{ij},j\neq i$ for (\ref{Eq:VectorizedRiccati}), solving (\ref{Eq:VectorizedRiccati}) still needs to deal with $P_i,i\in\mathcal{N}$. However, the Kalman equation (\ref{Eq:KuceraIdentityGeneral2}) produces linear equations that depend only on $Q_i,R_{ii},i\in\mathcal{N}$.

Note that a result similar to (\ref{Eq:VectorizedRiccati}) is presented in \cite[Lemma~2]{inga2019solution}, which removes the dependency on $P_i,i\in\mathcal{N},$ by using a stringent assumption that $(I_n\otimes B_i')$ is invertible. However, in the Kalman equation, we do not require such assumption.
\end{remark}

\section{Conclusions}
In this paper, we have answered the fundamental question: When is a given profile of strategies $\mathbf{K}^\dag$ Nash for a non-cooperative differential game $(A,[B_1,B_2,\cdots,B_N])$? The answer is characterized by a necessary and sufficient condition posed in the frequency-domain representation of the system. The condition provides a workaround to the inverse problem without numerically solving the convex feasibility problem. The Kalman equation reduces the redundancy in the coupled Riccati equation derived in the state-space representation. Future work lies around demonstrating the theories using application-driven examples.
\appendix

\subsection{Lemmas}
\begin{lemma}\label{Lemma:SpectralFactorization}
Let $\Pi_i(s)$ be defined in (\ref{Eq:DefinePhi}) for $i\in\mathcal{N},$ and the polynomial rank of $\Phi_i(s)$ is $p_i$. Suppose that Assumption \ref{Assum:StratStable} and the circle criterion (\ref{Eq:CircleCriterion}) holds. Then, there exists a $p_i \times m_i$ polynomial matrix $N_i(s)$ such that 
\begin{equation}\label{Eq:SpectralFactorizationPhi}
\Phi_i(s) = N_i'(s) N_i(s)
\end{equation}
with the rank of $N_i(s)$ is equal to $p_i$ for all $s$ such that $\mathfrak{Re}(s)>0$.
\end{lemma}

\begin{proof}
From \cite{callier1985polynomial}, we know that when $p_i = m_i$, there exist a spectral factorization satisfy (\ref{Eq:SpectralFactorizationPhi}). Now assume that $p_i<m_i$. Then, there exists a unimodular matrix $L_i(s)$ such that
$$
\Phi_i(s)L_i(s) = \begin{bmatrix}
\tilde{\Phi}_i(s) & 0
\end{bmatrix},
$$
where $\tilde{\Phi}_i$ is an $m_i\times p_i$ polynomial matrix with rank $p_i$. By definition (\ref{Eq:DefinePhi}), $\Phi(-s)_i'= \Phi_i(s)$. Obviously,
$$
L_i'(-s)\Phi_i(s) = L_i(-s)\Phi_i'(-s) = \begin{bmatrix}
\tilde{\Phi}_i'(-s)\\ 0
\end{bmatrix}.
$$
Post-multiplying it by $L_i(s)$ yields 
$$
L_i'(-s)\Phi_i(s)L_i(s) = \begin{bmatrix}
\hat{\Phi}_i(s) & 0\\
0 & 0
\end{bmatrix}
$$
where $\hat{\Phi}$ is a $p_i\times p_i$ polynomial matrix who has full rank because $\Phi_i(s)$ has rank $p_i$, and $L_i(s)$ is unimodular. Since the circle criterion holds, $\hat{\Phi}_i(jw)\geq 0$ for all $w\in\mathbb{R}$. Hence, by \cite{callier1985polynomial}, we can find a $p_i\times p_i$ polynomial matrix $\hat{N}_i(s)$ such that 
$$
\hat{\Phi}_i(s) = \hat{N}_i'(-s)\hat{N}_i(s)
$$
where $\hat{N}(s)$ has rank $p_i$ for all $s$ such that $\mathfrak{Re}(s)>0$. Therefore, we can construct 
$$
N_i(s)\coloneqq \begin{bmatrix}\hat{N}_i(s) & 0
\end{bmatrix}L_i^{-1}(s)
$$
such that the spectral factorization (\ref{Eq:SpectralFactorizationPhi}) exist.

\end{proof}










\bibliography{references}

\bibliographystyle{IEEEtran}

\end{document}